\documentclass[11pt]{amsart}
\usepackage{amscd,amssymb}
\usepackage{amsthm,amsmath,enumerate,amssymb,wasysym}
\usepackage[matrix,arrow]{xy}
\usepackage{enumerate}
\usepackage{longtable}
\usepackage{comment}
\usepackage{fancyhdr}
\usepackage{supertabular}
\usepackage{setspace}
\usepackage{color}
\usepackage{longtable}

\theoremstyle{definition}

\newtheorem*{example*}{Example}
\newtheorem{definition}[equation]{Definition}
\newtheorem{theorem}[equation]{Theorem}
\newtheorem{lemma}[equation]{Lemma}
\newtheorem{corollary}[equation]{Corollary}
\newtheorem{proposition}[equation]{Proposition}

\newtheorem*{conjecture*}{Conjecture}

\newtheorem*{question*}{Question}

\newtheorem*{problem*}{Problem}
\newtheorem*{theorem*}{Theorem}
\newtheorem*{Maintheorem*}{Main Theorem}

\theoremstyle{remark}
\newtheorem{remark}[equation]{Remark}
\newtheorem*{remark*}{Remark}

\newtheorem{claim}[equation]{Claim}

\makeatletter\@addtoreset{equation}{section} \makeatother
\renewcommand{\theequation}{\thesection.\arabic{equation}}

\author{Jihun Park and Joonyeong Won}

\title{Simply connected Sasaki-Einstein  rational homology $5$-spheres}

\begin{document}

\begin{abstract}
We completely determine which  simply connected rational homology $5$-spheres  admit Sasaki-Einstein metrics.
\end{abstract}


\address{ \emph{Jihun Park}\newline \textnormal{Center for Geometry and Physics, Institute for Basic Science
\newline \medskip 77 Cheongam-ro, Nam-gu, Pohang, Gyeongbuk, 37673, Korea \newline
Department of
Mathematics, POSTECH
\newline
77 Cheongam-ro, Nam-gu, Pohang, Gyeongbuk,  37673, Korea \newline
\texttt{wlog@postech.ac.kr}}}

\address{ \emph{Joonyeong Won}\newline \textnormal{Center for Geometry and Physics, Institute for Basic Science
\newline  77 Cheongam-ro, Nam-gu, Pohang, Gyeongbuk, 37673, Korea \newline
\texttt{leonwon@kias.re.kr}}}

\maketitle


\section{Introduction}
\label{section:intro}

A Riemannian manifold $(M,g)$  is called Sasakian if its conical metric $\bar{g}=r^2g+dr^2$ is a K\"ahler metric on the cone $C(M)=M\times\mathbb{R}^{+}$.  Sasakian metrics, which are defined on odd dimensional manifolds, can be considered as an odd dimensional counterpart  of K\"ahler metrics, which are defined on even dimensional manifolds. If the metric $g$ satisfies the Einstein condition, i.e., $\mathrm{Ric}_g=\lambda g$ for some constant $\lambda$, then the metric $g$ is called Einstein. It is well-known that a $(2n-1)$-dimensional Sasakian manifold can be Einstein only for $\lambda=2(n-1)$. Furthermore, a Sasakian metric $g$ is Einstein if and only if the conical metric $\bar{g}$ is Ricci-flat, i.e., $\mathrm{Ric}_{\bar{g}}=0$. 
The Sasakian manifold $M$ is isometrically embedded into $C(M)$ by~$M=M\times\{1\}\hookrightarrow C(M)$.  The cone $C(M)$ is equipped with an integrable complex structure $J$ since it is K\"ahler. The canonical vector field $r\partial_r$ defines the Reeb vector field $\xi$ on $M$ through the integrable complex structure, i.e.,
$\xi:=J(r\partial_r).$
Sasakian manifolds can be classified into three types according to the Reeb foliation $\mathcal{F}_\xi$ given by the Reeb vector field $\xi$. If the orbits of  the Reeb vector field $\xi$ are all closed, then $\xi$ integrates to an isometric $S^1$-action on $M$. Since $\xi$ vanishes nowhere, the action is locally free. If the action is free, then the Sasakian structure is said to be regular. If not, then it is said to be quasi-regular. On the other hand, if the orbits of  the Reeb vector field $\xi$ are not all closed, then it is said to be irregular. In the regular or the quasi-regular case, the space of leaves of the Reeb foliation $\mathcal{F}_\xi$ is a compact K\"ahler manifold  or orbifold, respectively. Furthermore, if $M$ is Sasaki-Einstein, then it becomes a K\"ahler-Einstein manifold  or orbifold.
Indeed, the classification of $(2n-1)$-dimensional quasi-regular Sasaki-Einstein manifolds is closely related to  the study of $(n-1)$-dimensional K\"ahler-Einstein Fano orbifolds (\cite[Proposition~7.5.33]{BoyerBook}). 

It is not an easy task to determine whether a given Fano orbifold admits an orbifold K\"ahler-Einstein metric. However,
the seminal work of Chen, Donaldson, Sun and Tian (\cite{CDS15}, \cite{TianYTD}) on existence of K\"ahler-Einstein metrics on Fano manifolds and their K-stability has  opened wide  a  new gate to an area where existence of K\"ahler-Einstein metrics  can be determined in purely algebraic ways. Since then, the result has been gradually being developed toward log $\mathbb{Q}$-Fano varieties (\cite{Berman}, \cite{LTW17},  \cite{LTW19}, \cite{LWX14},  \cite{SSY16}, \cite{TW}). Indeed, Li, Tian and Wang proved in \cite{LTW19} that the result of Chen, Donaldson, Sun and Tian
also holds for  log $\mathbb{Q}$-Fano varieties with  a mild assumption.
The following theorem is a simplified version of their result that allows us to immediately utilize it for our purpose.

\begin{theorem}[\cite{LTW19}]
\label{theorem:CDS}
Let $S$ be a del Pezzo surface with quotient singularities and $D$ be a prime divisor on $S$. Suppose that $-(K_S+\frac{m-1}{m}D)$ is ample for a positive integer $m$.
If $(S, \frac{m-1}{m}D)$ is uniformly  K-stable, then $S$ has a  K\"ahler-Einstein edge metric with angle $\frac{2\pi}{m}$ along~$D$.
\end{theorem}


Even though the theorem translates existence of K\"ahler-Einstein metrics  into an algebraic condition,  this algebraic condition is still extremely difficult to check in explicit  cases.
There are a few algebro-geometric  methods known to us that can  verify K-stability in concrete cases. The $\alpha$-invariant originally introduced by Tian (\cite{Tian87}) is one of the ways. 
The original definition of the $\alpha$-invariant was given  in an analytic way.  There is however an algebro-geometric   way to define the $\alpha$-invariant over an arbitrary field of characteristic zero.
 \begin{definition}
Let $(X, \Delta)$ be a log $\mathbb{Q}$-Fano variety. The $\alpha$-invariant of $(X, \Delta)$
is defined by the number
\[\alpha(X, \Delta)=\sup\left\{\lambda\in\mathbb{Q}\ \left|%
\aligned
&\text{the log pair}\ \left(X, \Delta+\lambda D\right)\ \text{is log canonical for every} \\
&\text{effective $\mathbb{Q}$-divisor $D$ numerically equivalent to $-(K_X+\Delta)$.}\\
\endaligned\right.\right\}.%
\]
\end{definition}

The $\alpha$-invariant plays a role in K\"ahler geometry by giving a sufficient condition for existence of orbifold K\"ahler-Einstein metrics.
\begin{theorem}[{\cite{DeKo01}, \cite{Nadel}, \cite{Tian87}}]
\label{theorem:alpha}
Let $(X,\Delta)$ be a Fano orbifold.
If $$\alpha(X, \Delta)>\frac{\dim(X)}{\dim(X)+1},$$ then $(X,\Delta)$ admits an orbifold K\"ahler-Einstein metric.
\end{theorem}
It quite often occurs that the  $\alpha$-invariant cannot determine existence of  an orbifold K\"ahler-Einstein metric on  a given Fano orbifold.

Recently Fujita and Odaka  introduced  a new algebro-geometric  way to test K-stability of log $\mathbb{Q}$-Fano varieties.  
Due to the works \cite{Berman}, \cite{CDS15}, \cite{LTW17},  \cite{LTW19}, \cite{LWX14},  \cite{SSY16}, \cite{TianYTD} and \cite{TW}, this supplies another method to check existence of orbifold K\"ahler-Einstein metrics.

To explain the method of Fujita and Odaka, let $(X, \Delta)$ be a $\mathbb{Q}$-factorial log pair with Kawamata log terminal singularities,
$Z\subset X$ a closed subvariety and $D$ an effective
 $\mathbb{Q}$-divisor on $X$. The log canonical threshold
of $D$ along $Z$ on the log pair $(X, \Delta)$  is the number given by 
$$c_Z(X,\Delta; D)=\mathrm{sup}\left\{\lambda\ \Big|\ \mbox{the log pair } (X, \Delta+\lambda D)
 ~\text{is log canonical  along}~Z.\right\}.$$
Since log canonicity is a local property, 
$$c_Z(X,\Delta; D)=\inf_{p\in Z} \left\{c_p(X,\Delta; D)\right\}.$$
If $X=\mathbb{C}^n$, $\Delta=0$, and $D=(f=0)$, where $f$ is a polynomial defined over $\mathbb{C}^n$, then we also
 use the notation $c_0(f)$ for the log canonical threshold of
 $D$ at the origin, instead of $c_0(X, 0; D)$.

\begin{definition}[\cite{BlumThesis}, \cite{Blum-Jonsson17}, \cite{Fujita-Odaka16}]
Let $(X, \Delta)$ be a log $\mathbb{Q}$-Fano variety and let $m$ be a positive integer such that the plurianticanonical linear system $|-m(K_X+\Delta)|$ is non-empty. Set $\ell_m=h^0(X,\mathcal{O}_X(-m(K_X+\Delta)))$.  For a section $s$ in $\mathrm{H}^0(X,\mathcal{O}_X(-m(K_X+\Delta)))$, we denote the effective divisor of the section $s$ by $D(s)$.
If $\ell_m$ sections $s_1,\ldots, s_{\ell_m}$  form  a basis of  the space $\mathrm{H}^0(X,\mathcal{O}_X(-m(K_X+\Delta)))$, then the   
$\mathbb{Q}$-divisor 
\[D:=\frac{1}{\ell_m}\sum_{i=1}^{\ell_m}\frac{1}{m}D (s_i)\]
is said to be of $m$-basis type with respect to the log $\mathbb{Q}$-Fano variety $(X,\Delta)$.
For a positive integer $m$, we set 
\[\delta_m(X,\Delta)=\inf_{\underset{\mbox{$m$-basis type}}{D:}}c_X(X,\Delta;D).\]
We set $\delta_m(X, \Delta)=0$ if $|-m(K_X+\Delta)|$ is empty.
The $\delta$-invariant of $(X,\Delta)$ is defined by the number
\[\delta(X,\Delta)=\limsup_m \delta_m(X,\Delta).\]
\end{definition}

The $\delta$-invariant turns out to provide  a necessary and sufficient criterion for uniform K-stability.
\begin{theorem}[{\cite[Theorem~D]{BlumThesis}, \cite{Blum-Jonsson17}, \cite{Fujita-Odaka16}}]
\label{theorem:delta}
Let $(X,\Delta)$ be a log $\mathbb{Q}$-Fano variety.
Then~$(X, \Delta)$ is uniformly K-stable if and only if $\delta(X,\Delta)>1$.
\end{theorem}

This potent criterion  has been put into practice  for smooth del Pezzo surfaces in  \cite{CRZ18}, \cite{CZ}, \cite{ParkWon},
and therein its effectiveness has been presented.

The development of the theory on quasi-regular Sasaki-Einstein metrics has followed that of the theory on K\"ahler-Einstein metrics on Fano varieties.
Indeed, since the $\alpha$-invariant method was adapted for Fano orbifolds by Demailly and Koll\'ar in \cite{DeKo01}, numerous  K\"ahler-Einstein Fano orbifolds have been detected, in particular, in dimension $2$. Such K\"ahler-Einstein Fano orbifolds  could yield  many examples of Sasaki-Einstein manifolds using the method introduced by Kobayashi (\cite[Theorem~5]{Kob63}) and fully matured by Boyer, Galicki and Koll\'ar (\cite{BG00}, \cite{BGK05}).

Now we have been strongly reinforced by new technologies for detecting K\"ahler-Einstein Fano orbifolds, in particular, the $\delta$-invariant method, so it would be natural to expect that many hidden Sasaki-Einstein manifolds can be detected by the new methods.
Indeed,  the  classification of  simply connected Sasaki-Einstein rational homology 5-spheres can be completed by applying the $\delta$-invariant method to certain hypersurfaces in  $3$-dimensional weighted projective spaces.

The main result of the present article is the complete classification of  simply connected Sasaki-Einstein  rational homology 5-spheres. Before we state the Main Theorem, let us explain how  closed simply connected spin $5$-manifolds are classified in \cite{Smale62}.

\begin{theorem}[\cite{Smale62}] For a positive integer $m$, there is a unique closed simply connected $5$-dimensional manifold $M_m$ with $\mathrm{H}_2(M_m,\mathbb{Z})=\mathbb{Z}/m\mathbb{Z}\oplus \mathbb{Z}/m\mathbb{Z}$ that admits a spin structure. Furthermore, 
a closed simply connected $5$-dimensional manifold $M$ that admits a spin structure  is of the form
\[M=k(S^2\times S^3)\#M_{m_1}\#\ldots\# M_{m_r},\]
where  $k(S^2\times S^3)$ is the $k$-fold connected sum of $S^2\times S^3$ for a  non-negative integer $k$ and $m_i$ is a positive integer greater than 1 with $m_i$ dividing $m_{i+1}$.
\end{theorem}
We denote by $kM_m$ the $k$-fold connected sum of $M_m$. Since a simply connected Sasaki-Einstein manifold must be spin (\cite[Proposition~2.6]{BGN03c}), Smale's classification of simply connected  $5$-manifolds will be enough for our purpose (cf. \cite{Barden}).

We  are now ready to state our main result.
\begin{Maintheorem*}
For each positive integer $n\geq 4$, the rational homology $5$-sphere $nM_2$ admits a Sasaki-Einstein metric.
\end{Maintheorem*}
Together with the works of Boyer, Galicki,  Koll\'ar and Nakamaye (\cite{BoNa10}, \cite{Ko05}, \cite{Ko06}), the Main Theorem completes the classification of  simply connected rational homology $5$-spheres that admit Sasaki-Einstein metrics.
\begin{theorem}
A simply connected rational homology $5$-sphere admits a (quasi-regular) Sasaki-Einstein metric if and only if it  is one of the following:
\begin{enumerate}
\item the $5$-sphere $S^5$;
\item $M_r$, where $r$ is a positive integer with $r\geq 2$ not divisible by $30$;
\item $2M_5$;
\item $2M_4$;
\item $2M_3, 3M_3, 4M_3$;
\item $nM_2$, where $n\geq 2$.
\end{enumerate}
\end{theorem}
\begin{proof}
It is a consequence of  \cite[Theorems~1.4 and 1.6]{Ko05} that a simply connected  Sasaki-Einstein rational  homology $5$-sphere must be one of the manifolds listed in the statement. Conversely, existence of Sasaki-Einstein metrics on  the rational homology $5$-spheres in the list is verified as follows.
\begin{enumerate}
\item The standard metric on the sphere is a Sasaki-Einstein metric.
\item See  \cite[the remark in Proof~9.6, Example~8.6 (1), (2), (3), (4)]{Ko05} and \cite[Theorem~2]{BG06}.
\item  This follows from \cite[Example~8.6 (5)]{Ko05}. 
\item  This also follows from \cite[Example~8.6 (5)]{Ko05}.
\item See \cite[Example~19]{Ko06} for $2M_3$. See \cite[Example~8.6 (5)]{Ko05}  for $3M_3$ and $4M_3$.
\item The assertion  \cite[Theorem~A]{BoNa10} verifies existence of Sasaki-Einstein metrics on $2M_2$, $3M_2$, $5M_2$, $6M_2$ and $7M_2$. 
On the other hand, the Main Theorem confirms existence of Sasaki-Einstein metrics on  $nM_2$ for every $n\geq 4$.
\end{enumerate}
\end{proof}

\begin{remark}
Regular Sasaki-Einstein metrics on simply-connected 5-manifolds are completely classified (\cite{FK89}). In particular, the  $5$-sphere is the only simply-connected regular Sasaki-Einstein rational homology $5$-sphere.  No irregular Sasaki-Einstein structure exists on a simply connected rational homology $5$-sphere (\cite{Suess}). 
\end{remark}

\section{Strategy for the proof of the Main Theorem}

The proof of the Main Theorem is based on the method introduced by Kobayashi and  developed  by Boyer, Galicki and Koll\'ar.
Our new ingredient  added to this method is to use the $\delta$-invariant instead of the $\alpha$-invariant.  Even though it is difficult to compute or estimate both the invariants in general, a few  methods have been developed well enough so that $\delta$-invariants can be estimated effectively on surfaces with at worst quotient singularities.

Let $X$ be a quasi-smooth hypersurface  in a weighted projective space $\mathbb{P}(\mathbf{w})=\mathbb{P}(a_0, a_1, \ldots, a_n)$ defined by a quasi-homogeneous polynomial $F(z_0,z_1,\ldots, z_n)$ in variables
$z_0, \ldots, z_n$ with weights $\mathrm{wt}(z_i)=a_i$. The equation $F(z_0,z_1,\ldots, z_n)=0$ also defines a hypersurface $\widehat{X}$ in $\mathbb{C}^{n+1}$ smooth outside the origin.  The link of $X$ is defined by the intersection 
\[L_X=S^{2n+1}_\mathbf{w}\cap \widehat{X},\]
where $S^{2n+1}_\mathbf{w}$ is the unit sphere centred at the origin in $\mathbb{C}^{n+1}$ with the Sasakian structure induced from the weight $\mathbf{w}=(a_0, a_1, \ldots, a_n)$ (see \cite[\S~1]{BG01} \cite[Example]{Takahashi}).
This is  a smooth compact  manifold of dimension $2n-1$.  It is simply-connected if $n\geq 3$ (\cite[Theorem~5.2]{M68}).
The situation can be diagrammed as follows (\cite{BoNa10}):
$$
\xymatrix{
L_X\ar@{->}[d]\ar@{^{(}->}[rr]&&S^{2n+1}_{\mathbf{w}}\ar@{->}[d]\\%
X\ar@{^{(}->}[rr]&&\mathbb{P}(\mathbf{w})}
$$
where the horizontal arrows are Sasakian and K\"ahlerian embeddings, respectively, and the vertical arrows are $S^1$ orbibundles and orbifold Riemannian submersions. 

Put $m=\mathrm{gcd}(a_1,\ldots, a_n)$. Suppose that $m>1$ and $\mathrm{gcd}(a_0, a_1, \ldots, a_{i-1}, \widehat{a_i}, a_{i+1}, \ldots, a_n)=1$ for each  $i=1,\ldots, n.$ Also set $b_0=a_0$ and $b_i=\frac{a_i}{m}$ for $i=1,\ldots, n.$  We also suppose that $\deg_{\mathbf{w}}(F)-\sum a_i<0.$ In other words, $X$ is a Fano orbifold.

There is a quasi-homogeneous polynomial $G(x_0,\ldots, x_n)$ in variables
$x_0, \ldots, x_n$ with weights $\mathrm{wt}(x_i)=b_i$ such that $F(z_0,z_1,\ldots, z_n)=G(z_0^d,z_1,\ldots, z_n).$ The equation $G(x_0,\ldots,x_n)=0$ defines a well-formed quasi-smooth hypersurface $Y$ in $\mathbb{P}(b_0, b_1, \ldots, b_n)$.    Denote by $D$ the divisor on $Y$ cut by $x_0=0$. 

\begin{lemma}[{\cite{BG00}, \cite{BGN03c}, \cite{BGK05}, \cite[Theorem~5]{Kob63}}]\label{lemma:BGK-method}
If there is a K\"ahler-Einstein edge metric on $Y$ with angle $\frac{2\pi}{m}$ along the divisor  $D$, then there is a Sasaki-Einstein metric on the link $L_X$ of $X$.
\end{lemma}

We now consider  a specific quasi-smooth hypersurface $X_n$ of degree $4n+2$ in $\mathbb{P}(2,2, 2n, 2n+1)$, where $n$ is a positive integer. We use quasi-homogeneous coordinates $x$, $y$, $z$, $w$ with weights $\mathrm{wt}(x)=\mathrm{wt}(y)=2$, $\mathrm{wt}(z)=2n$ and $\mathrm{wt}(w)=2n+1$.
By suitable coordinate changes,
$X_n$ may be assumed to be  given by
\[
w^2-z^2x-zr_{n+1}(x,y)-r_{2n+1}(x,y)=0,
\]
where $r_{n+1}$ and $r_{2n+1}$ are homogeneous polynomials of degrees $n+1$ and $2n+1$, respectively, in the variables $x, y$. Note that 
either $r_{n+1}$ contains $y^{n+1}$ or $r_{2n+1}$ contains $y^{2n+1}$  due to the quasi-smoothness of $X_n$.

Let $Y_n$ be the hypersurface  in $\mathbb{P}(1,1, n, 2n+1)$ defined by 
\[
w-z^2x-zr_{n+1}(x,y)-r_{2n+1}(x,y)=0,
\]
where we use the same notation for quasi-homogeneous coordinates as in $\mathbb{P}(2,2, 2n, 2n+1)$, abusing the notation.
Let $C_w$ be the curve in $Y_{n}$ that is cut out by the equation $w=0$.
Then the curve $C_{w}$ is reduced and irreducible. The log pair  \begin{equation}\label{the pair} \left(Y_n,\frac{1}{2}C_w\right)\end{equation}
is a log del Pezzo surface that works for the Main Theorem.

\begin{lemma}\label{lemma:Orlik}
The link of the surface  $X_n$ is $nM_2$.
\end{lemma}
\begin{proof}
This immediately follows from \cite[Theorem~7.1]{BGN03b}  and  \cite[Corollary]{MO70}.  The former is a reformulation  of the results of Savel'ev in \cite{Sav79}. The assertion also follows from \cite[Theorem~5.7]{Ko05}. Indeed, the genus of the curve~$C_w$ is $n$ and the Picard rank of $Y_n$ is $1$.
\end{proof}

It has long been known that \eqref{the pair} is a candidate that yields a Sasaki-Einstein metric on $nM_2$ (\cite[Example~10.3.7 and Open Problem~11.4.1]{BoyerBook}, \cite[Example~19]{Ko06}). The reason why this candidate had not been able to be confirmed as a Sasaki-Einstein metric producer on  $nM_2$   is that we did not have any method to determine whether~$\left(Y_n,\frac{1}{2}C_w\right)$ admits an orbifold K\"ahler-Einstein metric. In particular, the $\alpha$-invariant method is not sharp enough to do this job. Indeed, $\alpha\left(Y_n,\frac{1}{2}C_w\right)$ is at most $\frac{2}{3}$, which is too small to apply Theorem~\ref{theorem:alpha}.  However, the $\delta$-invariant  is decisive, so  that it allows  us to determine  existence of orbifold K\"ahler-Einstein metric on $\left(Y_n,\frac{1}{2}C_w\right)$ through its uniform K-stability.

It follows from Lemmas~\ref{lemma:BGK-method}~and~\ref{lemma:Orlik} that for the proof of the Main Theorem it is enough to show that $(Y_n, \frac{1}{2}C_w)$ possesses an orbifold K\"ahler-Einstein metric.  Theorems~\ref{theorem:CDS} and \ref{theorem:delta} then imply that the following assertion completes the proof of the Main Theorem.
\begin{theorem}\label{theorem:main-computation}
For each $n\geq 4$, $$\delta\left(Y_n,\frac{1}{2}C_w\right)\geq\frac{8n+8}{8n+7}.$$
\end{theorem}
This will be verified in Section~\ref{section:delta-computing}.

\section{Preliminaries}
\label{section:basic-tool}

Let $f$ be a polynomial over $\mathbb{C}$ in variables $z_1,\ldots, z_n$. Assign  integral weights $w(z_i)$ to the variables $z_i$. Let $w(f)$ be the weighted multiplicity of $f$ at the origin, i.e., the lowest weight of the monomials occurring in $f$,  and let $f_w$ denote the weighted homogeneous leading term of $f$, i.e., the term of the monomials  in $f$ with the  weighted multiplicity of $f$.

 Let $g$ be a polynomial over $\mathbb{C}$ in  $z_2,\ldots, z_n$ and set
\[h(z_1,\ldots,z_n)=f(z_1+g(z_2,\ldots, z_n), z_2,\ldots, z_n).\]
It is clear that \[h_w(z_1,\ldots,z_n)=f_w(z_1+g(z_2,\ldots, z_n), z_2,\ldots, z_n)\]
if $z_1+g(z_2,\ldots, z_n)$ is quasi-homogeneous with respect to the given weights $w(z_1),\ldots, w(z_n)$.
Let $f_1,\ldots, f_\ell$ be polynomials  over $\mathbb{C}$  in  $z_1,\ldots, z_n$.  We easily see that 
\[\left(\prod_{i=1}^{\ell}f_i\right)_w=\prod_{i=1}^{\ell}(f_i)_w, \  \ \  w\left(\prod_{i=1}^{\ell}f_i\right)=\sum_{i=1}^{\ell}w(f_i)\]
with respect to the given weights $w(z_1),\ldots, w(z_n)$.

\begin{lemma}\label{c_0}
Let $f$ be a polynomial over $\mathbb{C}$ in variables $z_1,\ldots, z_n$. Assign integral
weights $w(z_i)$ to the variables~$z_i$ and let $w(f)$ be the weighted multiplicity of $f$.  Then
\begin{enumerate}
\item $c_0(f_w)\leq c_0(f)\leq \frac{\sum w(z_i)}{w(f)}$.
\item  If  the log pair $$\left(\mathbb{C}^n, \frac{\sum w(z_i)}{w(f)}\cdot \left(f_w=0\right)\right)$$ is log canonical outside the origin, then 
$ c_0(f)=\frac{\sum w(z_i)}{w(f)}$.
\end{enumerate}
\end{lemma}
\begin{proof}
See \cite[Propositions~8.13 and 8.14]{Ko97}.
\end{proof}

Let $S$ be a surface with at most quotient singularities. We consider
 an irreducible and reduced curve $C$ on $S$ and  a point  $p$  of  the curve $C$. Let $D$ be an effective $\mathbb{Q}$-divisor on the surface~$S$.
We present here a few   well-known results concerning  log canonical singularities of the log pair~$(S, D)$.

\begin{lemma}
\label{lemma:mult-1}
Suppose that $p$ is a smooth point of the surface $S$.
If the log pair $(S,D)$ is not log canonical at $p$, then $\mathrm{mult}_p(D)>1$.
\end{lemma}
\begin{proof}
This is a well-known fact. See \cite[Proposition~9.5.13]{Laz-positivity-in-AG}, for instance.
\end{proof}
 
This immediately implies the following.

\begin{corollary}
\label{corollary:mult-1}
Suppose that $p$ is a smooth point of the surface $S$ and  the log pair $(S,D)$ is not log canonical at $p$.
If $C$ is not contained in the support of the divisor $D$,
then $D\cdot C>1$.
\end{corollary}

In general, the curve $C$ may be contained in the support of the divisor $D$. We write
$$
D=rC+\Delta,
$$
where $r$ is a non-negative rational number and $\Delta$ is an effective $\mathbb{Q}$-divisor on $S$ whose support does not contain the curve $C$.

\begin{lemma}
\label{lemma:inversion-of-adjunction}
Suppose that the log pair $(S,rC+\Delta)$ with $r\leq 1$ is not log canonical at $p$. If the surface $S$  and the curve $C$ are smooth at  $p$, then $$
C\cdot\Delta\geq\big(C\cdot\Delta\big)_p>1,
$$
where $\big(C\cdot\Delta\big)_p$ is the local intersection number of $C$ and $\Delta$ at $p$.
\end{lemma}

\begin{proof}
This immediately follows from Inversion of adjunction (see \cite[Theorem~7.5]{Ko97}, for instance).
\end{proof}

Let $\phi\colon\widetilde{S}\to S$ be the blow up at a smooth point $p$
and let~$E$ be the exceptional curve of the morphism $\phi$.
Then $\widetilde{S}$ has at most  quotient singularities and we have
$$
K_{\widetilde{S}}=\phi^*\left(K_S\right)+E.
$$
Denote by $\widetilde{D}$ the proper transform of the divisor $D$ via $\phi$. Then
$$
\widetilde{D}=\phi^*\left(D\right)-mE
$$
for some  non-negative rational number $m$ and
$$
K_{\widetilde{S}}+\widetilde{D}+\left(m-1\right)E=\phi^*\left(K_S+D\right).
$$
The log pair $(S,D)$ is log canonical at $p$ if and only if the log pair
$$
\left(\widetilde{S},\widetilde{D}+\left(m-1\right)E\right)
$$
is log canonical along the curve $E$.

In the present article, we deal with surfaces with quotient singularities. However, the statements mentioned so far require smoothness of the ambient space for us to utilize them to the fullest. Fortunately, the following assertion enables us to apply the same statements without any obstruction since our case has  a natural  finite morphism of a germ of the origin in $\mathbb{C}^2$ to a germ of a quotient singularity that is ramified only at a point.
\begin{proposition}[{\cite[Proposition~1.7]{Reid}, \cite[Proposition~3.16]{Ko97}}]\label{proposition:lc-by-finite-morphism}
Let $\varphi: Y\to X$ be a finite morphism between normal varieties and assume that $\varphi$ is unramified outside a set of codimension two. Then
a log pair $(X, D)$ is log canonical (resp. Kawamata log terminal) if and only if the log pair $(Y, \varphi^*D)$ is
log canonical (resp. Kawamata log terminal).
\end{proposition}

So far, we considered only local properties of the divisor $D$ on the surface $S$.
These properties will be used later to prove Theorem~\ref{theorem:main-computation}.
However, Theorem~\ref{theorem:main-computation} has a global aspect, so  we will need some global properties  of    $\mathbb{Q}$-divisors of $m$-basis type.  The following is originally due to Fujita and Odaka (\cite[Lemma 2.2]{Fujita-Odaka16}).

\begin{lemma}[{\cite[Corollary~2.10]{CPS18}}]
\label{lemma:a-bound}
We now suppose that $(S, \Omega)$ is a log del Pezzo surface.  In addition,
suppose that~$D$ is an ample  $\mathbb{Q}$-divisor of $m$-basis  type with respect to  $(S, \Omega)$. Let $\Gamma$ be an integral curve on $S$ with
$
\Gamma\sim_{\mathbb{Q}} \mu D
$
for some positive rational number $\mu$. Then 
$$
\mathrm{mult}_\Gamma(D)\leq\frac{1}{3\mu}+\epsilon_m,
$$
where $\epsilon_m$ is a constant depending on $m$ such that $\epsilon_m\to 0$ as $m\to \infty$.
\end{lemma}

\section{Proof of Theorem~\ref{theorem:main-computation}}
\label{section:delta-computing}
For  convenience, denote by $\mathbb{P}$ the weighted projective space $\mathbb{P}(1,1,n,2n+1)$.
The quasi-smooth hypersurface $Y_{n}$ of degree $2n+1$ in $\mathbb{P}$  has a unique singular point  at the point $o_z=[0:0:1:0]$, which is  a cyclic quotient singularity of type $\frac{1}{n}(1,1)$.
We see that $$-\left(K_{Y_n}+\frac{1}{2}C_w\right)\sim_{\mathbb{Q}}\frac{3}{2}H,$$ where $H$ is a hyperplane section of weighted degree $1$.

With a sufficiently large  integer $m$, let $D$ be a $\mathbb{Q}$-divisor of $2mn$-basis type with respect to the log del Pezzo surface  $\left(Y_{n}, \frac{1}{2}C_w\right)$. Put $\lambda =\frac{8n+8}{8n+7}$.

\begin{theorem}\label{theorem:smooth-locus}
For $n\geq 4$, the log pair $\left(Y_{n}, \frac{1}{2}C_w+\lambda D\right)$  is log canonical outside the  singular point~$o_z$ of $Y_{n}$.\end{theorem}
\begin{proof}
Suppose that the log pair $\left(Y_{n}, \frac{1}{2}C_w+\lambda D\right)$ is not log canonical at a smooth point $p$.  There is a unique curve $C$ in $|\mathcal{O}_{Y_{n}}(1)|$ that passes through the point $p$. Note that both $C_w$ and $C$ are irreducible and reduced. We may write
\[D=aC+bC_w+\Delta,\]
where $a, b$ are non-negative rational numbers and $\Delta$ is an effective $\mathbb{Q}$-divisor whose support contains neither $C$ nor $C_w$. Note that both $C_w$ and $C$ are smooth at $p$. 
Due to Lemma~\ref{lemma:a-bound},  we may assume that 
\[a\leq \frac{27}{50}, \ \ \ b\leq  \frac{27}{50(2n+1)}.\]

If $p$ lies outside $C_w$, then $\left(Y_{n}, C+\lambda \Delta\right)$ is not log canonical at $p$. It then follows from Lemma~\ref{lemma:inversion-of-adjunction} that $\left(\Delta\cdot C\right)_p>\frac{1}{\lambda}$. However, this yields an absurd inequality
\[\frac{1}{\lambda}<\Delta\cdot C=\frac{\left(\frac{3}{2}-a-b(2n+1)\right)(2n+1)}{n(2n+1)}\leq \frac{3}{2n}\leq\frac{3}{8}.\]

If the curve $C$ transversally intersects $C _w$ at $p$, then Lemma~\ref{lemma:inversion-of-adjunction} implies $$1<\left(\left(\left(\frac{1}{2}+\lambda b\right)C_w+\lambda\Delta\right)\cdot C\right)_p= \frac{1}{2}+\lambda b+\lambda\left(\Delta\cdot C\right)_p\leq \frac{1}{2}+\lambda \left(\frac{27}{50(2n+1)}+\frac{3}{2n}\right).$$
This is absurd. Therefore,  the curve $C$ must intersect $C _w$ at $p$ tangentially.   We have
\[\mathrm{mult}_p(\Delta)\leq \Delta\cdot C\leq \frac{3}{2n}.\]

Let $\phi\colon\widetilde{Y}_{n}\to Y_{n}$ be the blow up at the point $p$ and let $E$ be the exceptional divisor of~$\phi$.  Denote by $\widetilde{C}_w$, $\widetilde{C}$ and $\widetilde{\Delta}$ the proper transforms of $C_w$, $C$ and $\Delta$, respectively.
We then obtain
\[\phi^*\left(K_{Y_{n}}+\left(\frac{1}{2}+\lambda b\right)C_w+\lambda aC+\lambda\Delta\right)=K_{\widetilde{Y}_{n}}+\left(\frac{1}{2}+\lambda b\right)\widetilde{C}_w+\lambda a\widetilde{C}+\lambda\widetilde{\Delta}+cE,\]
where $c=\lambda a+\lambda b+\lambda \mathrm{mult}_p(\Delta) -\frac{1}{2}$.  Since $c\leq 1$ and $\lambda \mathrm{mult}_p(\Delta)\leq 1$, the log pair $$\left(\widetilde{Y}_{n},\left(\frac{1}{2}+\lambda b\right)\widetilde{C}_w+\lambda a\widetilde{C}+\lambda\widetilde{\Delta}+cE\right)$$ is not log canonical at the point $q$ where $E$, $\widetilde{C}_w$ and $\widetilde{C}$ meet. 

Let $\psi\colon\overline{Y}_{n}\to \widetilde{Y}_{n}$ be the blow up at the point $q$ and let $F$ be the exceptional divisor of~$\psi$. 
Denote by $\overline{C}_w$, $\overline{C}$,  $\overline{\Delta}$ and $\overline{E}$ the proper transforms of $\widetilde{C}_w$, $\widetilde{C}$,  $\widetilde{\Delta}$ and $E$ by $\psi$,  respectively.
Then
\[\psi^*\left(K_{\widetilde{Y}_{n}}+\left(\frac{1}{2}+\lambda b\right)\widetilde{C}_w+\lambda a\widetilde{C}+\lambda\widetilde{\Delta}+cE\right)=
K_{\overline{Y}_{n}}+\left(\frac{1}{2}+\lambda b\right)\overline{C}_w+\lambda a\overline{C}+\lambda\overline{\Delta}+c\overline E+dF,\]
where $d=\lambda a+\lambda b+c+\lambda \mathrm{mult}_{q}(\widetilde{\Delta}) -\frac{1}{2}$. Since
$$d=2\lambda (a+b)+\lambda \left(\mathrm{mult}_p(\Delta)+\mathrm{mult}_{q}(\widetilde{\Delta})\right)-1\leq 2\lambda \left(a+b+\mathrm{mult}_p(\Delta)\right)-1\leq 1,$$ the log pair
$$\left(\overline{Y}_{n}, \left(\frac{1}{2}+\lambda b\right)\overline{C}_w+\lambda a\overline{C}+\lambda\overline{\Delta}+c\overline E+F\right)$$ is not log canonical at a point on $F$. Meanwhile, the curves $\overline{C}_w$, $\overline{C}$ and $\overline{E}$ intersect $F$ transversally at distinct points and $\lambda\overline{\Delta}\cdot F=\lambda \mathrm{mult}_{q}(\widetilde{\Delta})\leq 1$, and hence the log pair must be log canonical along~$F$. This is a contradiction.
Consequently, the original log pair $\left(Y_{n}, \frac{1}{2}C_w+\lambda D\right)$  must be log canonical outside~$o_z$.
\end{proof}
Before we proceed further, we compute the dimension of $\mathrm{H}^0\left(Y_n, \mathcal{O}_{Y_n}(3mn)\right)$.
\begin{lemma}
For each positive integer $m$, $$h^0\left(Y_{n},\mathcal{O}_{Y_n}(3mn)\right)=\frac{9}{2}m^2n+\frac{3}{2}mn+3m+1.$$
\end{lemma}
\begin{proof} 
From the exact sequence
\[0\longrightarrow \mathcal{O}_{\mathbb{P}}\left(3mn-(2n+1)\right)\longrightarrow \mathcal{O}_{\mathbb{P}}\left(3mn\right)\longrightarrow \mathcal{O}_{Y_n}\left(3mn\right)\longrightarrow 0\]
we obtain $$h^0\left(Y_{n},\mathcal{O}_{Y_n}(3mn) \right)= h^0\left(\mathbb{P},\mathcal{O}_{\mathbb{P}}(3mn) \right)-h^0\left(\mathbb{P},\mathcal{O}_{\mathbb{P}}(3mn-(2n+1))\right),$$ i.e., the wanted dimension is equal to the difference of  the number of monomials of degree $3mn$ and the  number of monomials of degree $3mn-(2n+1)$ in  $\mathbb{P}$.

Since $w=z^2x+zr_{n+1}(x,y)+r_{2n+1}(x,y)$  on $Y_n$, every monomial  of degree $3mn$ containing  $w$  can be expressed as a quasi-homogeneous polynomial of degree $3mn$ in the variables $x, y, z$. Therefore,   
 $\mathrm{H}^0(Y_n, \mathcal{O}_{Y_n}(3mn))$ is generated by  the monomials of degree $3mn$ that do not involve $w$.  
Since the weight of $z$ is $n$,  the set of monomials
\begin{equation}\label{equation:monomial-set}
\mathcal{S}=\left\{x^{n_1}y^{n_2} ~|~ n_1+n_2=(3m-j)n \text{ and } 0\leq j \leq 3m \right \}
\end{equation}
is in 1-1 correspondence with the set of the monomials spanning  $\mathrm{H}^0\left(Y_n, \mathcal{O}_{Y_n}(3mn)\right)$.
This implies the claim.
\end{proof}

\begin{theorem}\label{theorem:singular-point}
For $n\geq 4$, the log pair $\left(Y_{n}, \frac{1}{2}C_w+\lambda D\right)$  is log canonical at $o_z$.
\end{theorem}

\begin{proof} 
Put $\ell_m=h^0(Y_{n},\mathcal{O}_{Y_n}(3mn))$.  Denote by $v_m$ the positive integer such that 
$$\prod_{\mathbf{x}\in\mathcal{S}}\mathbf{x}=(xy)^{v_m},$$
where $\mathcal{S}$ is the set in \eqref{equation:monomial-set}.  Indeed, $$v_m=\frac{1}{4}nm(3m+1)(6nm+n+3)=mn\ell_m+\frac{1}{4}mn\left(3mn-3m+n-1\right).$$

Let $\{s_1,\ldots, s_{\ell_m}\}$ be a basis of  the vector space
$\mathrm{H}^0(Y_n, \mathcal{O}_{Y_n}(3mn)).$ Denote by $B_i$ the effective divisor of the section $s_i$ and set 
\[B:=\sum_{i=1}^{\ell_m}B_i.\]
Note that $\frac{1}{2mn\ell_m}B$ is  of $2mn$-basis type with respect to $\left(Y_{n}, \frac{1}{2}C_w\right).$

In the affine piece $U$ given by $z\ne 0$, the surface $Y_n$ is defined  by
\begin{equation}\label{local-equation}
w=x+r_{n+1}(x,y)+r_{2n+1}(x,y).
\end{equation}
In a neighborhood of  $o_z$, $y$ and $w$ may be regarded as local coordinates with  $\mathrm{wt}(y)=1$ and $\mathrm{wt}(w)=1$. However, instead of $y$ and $w$, we may regard $x$ and $y$ as local coordinates with  $\mathrm{wt}(x)=1$ and $\mathrm{wt}(y)=1$, due to \eqref{local-equation}. Even though $U$ is the  quotient  of $\mathbb{C}^3$ by the action $\zeta_n\cdot (x, y,w)\mapsto (\zeta_n x,  \zeta_n y,\zeta_n w)$, where $\zeta_n$ is a primitive $n$-th root of unity,
Proposition~\ref{proposition:lc-by-finite-morphism} allows us to replace  $Y_n$ by $\mathbb{C}^2$ with coordinates $x$ and $y$ and the point $o_z$ by the origin $(0,0)$. This also enables us to make use of the Newton polygon  method as in \cite{ParkWon}.

Each divisor $B_i$ can be defined in a neighborhood of the origin  by a polynomial $f_i$ of degree at most $3mn$ in  the variables~$x, y$. The divisor $B$ is defined by $f:=\prod f_i$ around the origin. 

Put $$g(x,y)=x+r_{n+1}(x,y)+r_{2n+1}(x,y).$$ Then $g=0$ defines the curve $C_w$ locally around the origin.  Set $h:=g^{mn\ell_m}f$.

Denote by $\mathcal{S}_{k}$  the set of monomials of degree $k$ in $x, y$. It then follows from  \cite[Lemma~4.3]{ParkWon} that there is an injective map 
$$I_m : \left\{f_i ~| ~1\leq i\leq \ell_m  \right\}\rightarrow \bigcup_{k=0}^{3mn}\mathcal{S}_{k}$$  such that  the monomial 
$I_m (f_i)$ is contained in $f_i$ for each $i$. The set 
$$\left\{ x^{n_1}y^{n_2}z^{3mn-n_1-n_2} \ | \ x^{n_1}y^{n_2}\in \mathcal{S}\right\}$$
 is linearly independent so that it should form a basis of   $\mathrm{H}^0(Y_{n},\mathcal{O}_{Y_n}(3mn) )$, and hence the image of  the map $I_m$ is exactly the set $\mathcal{S}$. 

Note that  the log pair $\left(Y_n,  \frac{1}{2}C_w+\frac{\lambda}{2mn\ell_m}B\right)$ is log canonical if $\left(Y_n,  \frac{\lambda}{2}C_w+\frac{\lambda}{2mn\ell_m}B\right)$ is log canonical.
Since  \[\begin{split}
 c_{o_z}\left(Y_n, \frac{1}{2}C_w+\frac{1}{2mn\ell_m}B\right)&=2mn\ell_m \cdot c_{o_z}\left(Y_n, mn\ell_m C_w+B\right)\\
 &=2mn\ell_m \cdot c_{0}\left(g^{mn\ell_m}f\right),\\
 \end{split}\]
in order to prove the statement, it is enough  to show that
\[c_0\left(h\right)\geq\frac{\lambda}{2mn\ell_m}.\]
This will be verified by a sophisticated version of the Newton polygon method  in~\cite{ParkWon} for log del Pezzo surfaces.
\medskip

We now consider the Newton polygon of the polynomial $f$. We use coordinate functions~$(s,t)$ for $\mathbb{R}^2$ in which the Newton polygon sits. The coordinate function $s$ corresponds to the exponents of $x$  appearing  in constituent monomials.
\medskip

\renewcommand{\theequation}{\arabic{equation}}
\setcounter{equation}{0}
\begin{claim}\label{claim:basis-mult}
The Newton polygon of the polynomial $f$ contains the point $(v_m, v_m)$ corresponding to the monomial~$x^{v_m}y^{v_m}$.
\end{claim}
\medskip

Since each $f_i$ contains the monomial $I_m (f_i)$, we have $w(f_i)\leq w(I_m (f_i))$ with respect to  given weights $w(x), w(y)$, and hence 
$$w(f)=w\left(\prod_{i=1} ^{\ell_m}f_i\right)\leq w\left(\prod_{i=1} ^{\ell_m}I_m \left(f_i\right)\right)=w\left(\prod_{\mathbf{x}\in\mathcal{S}}\mathbf{x}\right)=w(x^{v_m}y^{v_m}).$$ 
This proves the claim.
\medskip

\begin{claim}\label{claim:exponent-y}
The polynomial $g$ must contain  the monomials $x$ and $y^\nu$, where $\nu$ is a positive integer  less than or equal to  $2n+1$.
\end{claim}
Due to the quasi-smoothness of $Y_n$, $g$ must contain either $y^n$ or $y^{2n+1}$.

We keep $\nu$ for the lowest integer such that $y^\nu$ appears in $g$. Note that $\nu$ is either $n+1$ or~$2n+1$. We may assume that the coefficient of $y^{\nu}$ is $1$, so that we could write $$g(x,y)=x+y^{\nu}+\mbox{the remaining terms}.$$
The Newton polygon of $g$ has only two vertices. One is from $x$ and the other from  $y^\nu$.

\begin{claim}\label{claim:Newton}
For sufficiently large $m$, the Newton polygon of the polynomial $h$ must contain  the point 
 $$\left(\frac{2mn\ell_m}{\lambda}, \frac{2mn\ell_m}{\lambda}\right).$$ 
\end{claim}
It follows from Claim~\ref{claim:exponent-y} that the Newton polygon of $g^{mn\ell_m}$  contain the point 
 $\left(mn\ell_m\frac{2n+1}{2n+2}+1, mn\ell_m\frac{2n+1}{2n+2}+1\right),$
and hence the Newton polygon of $h$ contains the point 
 $$\left(mn\ell_m\frac{2n+1}{2n+2}+v_m+1, mn\ell_m\frac{2n+1}{2n+2}+v_m+1\right).$$ 
 Note that \[\begin{split}
 mn\ell_m\frac{2n+1}{2n+2}+v_m+1&=mn\ell_m\frac{2n+1}{2n+2}+\frac{1}{4}mn\left(4\ell_m+3mn-3m+n-1\right)+1\\
 &=mn\ell_m\frac{4n+3}{2n+2}+\frac{1}{4}mn\left(3mn-3m+n-1\right)+1.\\
 \end{split}\]
 Since $\ell_m=\frac{9}{2}m^2n+\frac{3}{2}mn+3m+1$, for sufficiently large $m$
 \[mn\ell_m\frac{2n+1}{2n+2}+v_m +1< mn\ell_m\frac{8n+7}{4n+4}=\frac{2mn\ell_m}{\lambda}.\]
 \medskip

 Let $\Lambda$ be the edge of the Newton polygon of $f$ that intersects the line $s=t$.  If the Newton polygon of $f$ meets the line $s=t$ at one of its vertices, then we choose the edge that meets the line $s=t$ and sits on the side of $s\geq t$ in $\mathbb{R}^2$.

\begin{claim}
If  $\Lambda$ is vertical, then $c_0\left(h\right)\geq \frac{\lambda}{2mn\ell_m}$. 
\end{claim}

Assign $w'(x)=\nu$ and $w'(y)=1$. 
Then, $g_{w'}=x+y^{\nu}$. On the other hand,
\[f_{w'}=\epsilon x^ay^b(x+y^\nu)^c(x+A_1y^\nu)^{c_1}\cdots(x+A_ry^\nu)^{c_r},\]
where $a$, $b$, $c$ and $c_i$ are non-negative integers, $A_i$ are non-zero constants other than $1$, and $\epsilon$ is a non-zero constant. 
Since $\Lambda$ is vertical, every monomial in $f_{w'}$ is plotted below the line $s=t$ in $\mathbb{R}^2$. The exponent $b$ cannot therefore  exceed~$v_m$  since the Newton polygon of $f$ contains the point $(v_m, v_m)$.  Also, $\nu a$ cannot exceed $(1+\nu)v_m$ either. The exponent $c$ and $c_i$ are at most $\frac{3mn\ell_m}{n+1}$ since $\nu c, \nu c_i \leq \deg (f)\leq 3mn\ell_m.$
Therefore, from Lemma~\ref{c_0} we obtain
$$c_0\left(h\right)\geq c_0\left(h_{w'}\right)=c_0\left(f_{w'}g_{w'}^{mn\ell_m}\right)=\mathrm{min}\left\{\frac{1}{a},\frac{1}{b},\frac{1}{c+mn\ell_m}, \frac{1}{c_1},\ldots, \frac{1}{c_r} \right\}\geq \frac{\lambda}{2mn\ell_m}.$$
\medskip

We now assume that $\Lambda$ is not vertical.  
Set  integral weights~$w(x), w(y)$ in  such a way  that   all the monomials of $f_w$ are plotted  on the edge~$\Lambda$. 
Then, the slope of the edge  $\Lambda$ is equal to $-\frac{w(x)}{w(y)}$.

\medskip

\textbf{Step A.}
We write an irreducible decomposition of $f_w$  as $$f_w=\epsilon x^ay^b\left(x^{\alpha_1}+g_1(x,y)\right)^{c_1} \cdots\left(x^{\alpha_r}+g_r(x,y)\right)^{c_r},$$ 
where $\epsilon$ is a non-zero constant and $g_i(x,y)$ is a quasi-homogeneous polynomial of degree $w(x^{\alpha_i})$ which does not contain the monomial $x^{\alpha_i}$.
Note that  $a,b\leq v_m$. 

\medskip

Let $c=\max\{c_i\}.$  We may assume that $c_1=c$.  For  convenience, we set $\alpha=\alpha_1$. Since~$x^\alpha+g_1(x,y)$ is irreducible,  $g_1(x,y)$ must  contain   the monomial $y^\beta$ for some positive integer $\beta$.

\medskip

\begin{claim}\label{claim:c-bound}
If $g_{i}(x, y)$ contains $y^\gamma$, then the exponent $c_i$ is at most $\frac{3mn\ell_m}{\gamma}$.
\end{claim}

This immediately follows from the inequality
\[3mn\ell_m\geq \deg(f_w)\geq c_i\gamma.\]

\medskip

Set $\sigma_m=3v_m-\frac{2mn\ell_m}{\lambda}$. Note that
\[v_m< 3v_m-\frac{2v_m}{\lambda}<\sigma_m\]
and for a sufficiently large $m$,
\[\sigma_m\leq \frac{2mn\ell_m}{\lambda}.\]

\begin{claim}\label{claim:small-c}
If $c\leq \sigma_m$, then $c_0\left(h\right)\geq \frac{\lambda}{2mn\ell_m}.$
\end{claim}

For the claim we consider the Newton polygon of $h$. If the Newton polygon of $h$ meets the line $s=t$ at one of its vertices, we may  choose weights $w'(x)$, $w'(y)$ in such a way that $$h_{w'}=x^{a'}y^{a'}.$$ Since the Newton polygon of $h$ contains the point $\left(\frac{2mn\ell_m}{\lambda}, \frac{2mn\ell_m}{\lambda}\right)$,  we have $a'\leq \frac{2mn\ell_m}{\lambda}$. Therefore,
$$c_0\left(h\right)\geq c_0\left(h_{w'}\right)=\frac{1}{a'}\geq\frac{\lambda}{2mn\ell_m}.$$ This allows us to assume that the line $s=t$ does not pass through any vertex point of the Newton polygon of $h$.

We assign new weights $w'(x)$, $w'(y)$ in such a way that $h_{w'}$ corresponds to the edge of the Newton polygon of $h$ intersecting the line~$s=t$. 

If the edge is either vertical or horizontal,  then $h_{w'}=x^{a'}\tilde{h}(y)$ or $y^{a'}\tilde{h}(x)$, respectively, where $\tilde{h}$ has multiplicity at most $a'-1$ at the origin. Since the Newton polygon of $h$ contains the point $\left(\frac{2mn\ell_m}{\lambda}, \frac{2mn\ell_m}{\lambda}\right)$,
we have $a'\leq \frac{2mn\ell_m}{\lambda}$. Therefore, $c_0\left(h\right)\geq c_0\left(h_{w'}\right)\geq \frac{\lambda}{2mn\ell_m}$. 

We may now assume that the edge is neither vertical nor horizontal. 

Depending on $g_{w'}$, we have the following three cases.

\medskip
Case (a). $g_{w'}=x+y^{\nu}$.
\medskip

In this case, an irreducible decomposition of $f_{w'}$ is given as
$$f_{w'}=\epsilon x^{a'}y^{b'}\left(x+A_1y^{\nu}\right)^{c'_1}\cdots\left(x+A_{r'}y^{\nu}\right)^{c'_{r'}},$$ where $A_i$ are non-zero distinct constants and $\epsilon$ is a non-zero constant. Therefore,
$$h_{w'}=\epsilon x^{a'}y^{b'}\left(x+A_1y^{\nu}\right)^{c'_1}\cdots\left(x+A_{r'}y^{\nu}\right)^{c'_{r'}}\left(x+y^{\nu}\right)^{mn\ell_m}.$$ Since the Newton polygon of $h$ contains the point $(\frac{2mn\ell_m}{\lambda}, \frac{2mn\ell_m}{\lambda})$,  we obtain $a', b'\leq \frac{2mn\ell_m}{\lambda}$.
If~$A_i\ne 1$ for any $i=1,\ldots, r'$, then Lemma~\ref{c_0} implies
$$c_0\left(h_{w'}\right)=\mathrm{min}\left\{\frac{1}{a'},\frac{1}{b'},\frac{1}{c'_i},\frac{1}{mn\ell_m},\frac{w(x)+w(y)}{w(h_{w'})} \left| \ i=1,\ldots, r' \right.\right\}.$$
If $A_i=1$ for some $i$, then Lemma~\ref{c_0} implies
$$c_0\left(h_{w'}\right)=\mathrm{min}\left\{\frac{1}{a'},\frac{1}{b'},\frac{1}{c_j'},\frac{1}{c_i'+mn\ell_m},\frac{w(x)+w(y)}{w(h_{w'})} \left| \ j\ne i \right. \right\}.$$
Claim~\ref{claim:c-bound} holds  for  arbitrary weights. Therefore it implies that $c_i'\leq \frac{3mn\ell_m}{\nu}\leq \frac{3mn\ell_m}{5}$ and hence 
$\frac{1}{c_i'+mn\ell_m}>\frac{\lambda}{2mn\ell_m}$.
Since the Newton polygon of $h$ contains the point $\left( \frac{2mn\ell_m}{\lambda}, \frac{2mn\ell_m}{\lambda}\right)$, we obtain $\frac{w(x)+w(y)}{w(h_{w'})}\geq \frac{\lambda}{2mn\ell_m}$.
Consequently, $c_0\left(h\right)\geq \frac{\lambda}{2mn\ell_m}.$

\medskip
Case (b). $g_{w'}=x$. 
\medskip

In this case, an irreducible decomposition of $f_{w'}$ is given as 
$$f_{w'}=\epsilon x^{a'}y^{b'}\left(x^{\kappa_1}+h_1(x,y)\right)^{c_1'}\left(x^{\kappa_2}+h_2(x,y)\right)^{c_2'} \cdots\left(x^{\kappa_r}+h_{r'}(x,y)\right)^{c_{r'}'},$$
where $\epsilon$ is a non-zero constant and $h_i(x,y)$ is a quasi-homogeneous polynomial of degree $w'(x^{\kappa_i})$ which does not contain the monomial $x^{\kappa_i}$.   Then
$$h_{w'}= \epsilon x^{a'}y^{b'}\left(x^{\kappa_1}+h_1(x,y)\right)^{c_1'}\left(x^{\kappa_2}+h_2(x,y)\right)^{c_2'} \cdots\left(x^{\kappa_r}+h_{r'}(x,y)\right)^{c_{r'}'}x^{mn\ell_m}.$$
From the fact that the Newton polygon of $h$ contains the point $\left( \frac{2mn\ell_m}{\lambda}, \frac{2mn\ell_m}{\lambda}\right)$, it follows that $a'+mn\ell_m$, $b'$ and $\frac{w(h_{w'})}{w(x)+w(y)}$  cannot exceed $\frac{2mn\ell_m}{\lambda}$.

Let $c'=\max\{c_i'\}.$

If $w$ and $w'$ define the same slope, then $c=c'$, and hence $c'\leq \sigma_m$. 

Suppose that $w$ and $w'$ define  different slopes and, in addition, that $c'>\sigma_m$. Let $\Lambda_{f_{w'}}$ be the edge defined by $f_{w'}$ and let $\Lambda_{h_{w'}}$ be the edge defined by $h_{w'}$. The edge $\Lambda_{h_{w'}}$ meets the line $s=t$ at an interior point of $\Lambda_{h_{w'}}$. 
On the other hand, $\Lambda_{f_{w'}}$ does not intersect  the line $s=t$ at an interior point of $\Lambda_{f_{w'}}$.
 Observe that $\Lambda_{h_{w'}}$ is the translation of $\Lambda_{f_{w'}}$ by $mn\ell_m$ along the $s$-axis. This means that
the edge $\Lambda_{f_{w'}}$   lies on the side $s\leq t$ in $\mathbb{R}^2$. However, the condition $c'>\sigma_m$ implies that the edge
$\Lambda_{f_{w'}}$ has a point $(s_0, t_0)$ with $s_0>\sigma_m>v_m$. The Newton polygon of $f$  cannot then contain the point $(v_m, v_m)$. This contradicts Claim~\ref{claim:basis-mult}. Consequently, $c'\leq \sigma_m$.
 
Lemma~\ref{c_0} then implies
$$c_0\left(h_{w'}\right)=\mathrm{min}\left\{\frac{1}{a'+mn\ell_m},\frac{1}{b'},\frac{1}{c'},\frac{w(x)+w(y)}{w(h_{w'})} \right\}\geq \frac{\lambda}{2mn\ell_m},$$
and hence $c_0\left(h\right)\geq \frac{\lambda}{2mn\ell_m}.$

\medskip
Case (c). $g_{w'}=y^\nu$.
\medskip

Similarly to Case (b), an irreducible decomposition of $f_{w'}$ is given as 
$$f_{w'}=\epsilon x^{a'}y^{b'}\left(x^{\kappa_1}+h_1(x,y)\right)^{c_1'}\left(x^{\kappa_2}+h_2(x,y)\right)^{c_2'} \cdots\left(x^{\kappa_r}+h_{r'}(x,y)\right)^{c_{r'}'},$$
where $\epsilon$ is a non-zero constant and $h_i(x,y)$ is a quasi-homogeneous polynomial of degree $w'(x^{\kappa_i})$ which  does not contain the monomial $x^{\kappa_i}$.   Then,
$$h_{w'}= \epsilon x^{a'}y^{b'}\left(x^{\kappa_1}+h_1(x,y)\right)^{c'_1}\left(x^{\kappa_2}+h_2(x,y)\right)^{c_2'} \cdots\left(x^{\kappa_r}+h_{r'}(x,y)\right)^{c_{r'}'}y^{\nu mn\ell_m}.$$
The arguments in Case~(b) work almost verbatim for Case~(c). The only difference is that $\Lambda_{h_{w'}}$ is the translation of $\Lambda_{f_{w'}}$ by $\nu mn\ell_m$ along the $t$-axis. But the difference does not damage the proof at all. We consequently obtain  from Lemma~\ref{c_0} that
$$c_0\left(h\right)\geq c_0\left(h_{w'}\right)=\mathrm{min}\left\{\frac{1}{a'},\frac{1}{b'+\nu mn\ell_m},\frac{1}{c'},\frac{w(x)+w(y)}{w(h_{w'})} \right\}\geq \frac{\lambda}{2mn\ell_m}.$$

The three cases above complete the proof of Claim~\ref{claim:small-c}.

\medskip

\begin{claim}\label{claim-horizontal}
If $\Lambda$ is horizontal, then $c_0\left(h\right)\geq \frac{\lambda}{2mn\ell_m}.$
\end{claim}
For the proof of Claim~\ref{claim-horizontal}, the  proof of Claim~\ref{claim:small-c} works without the condition $c\leq \sigma_m$. The condition is required only in Cases~(b) and (c).
More precisely, it is required only when $w$ and $w'$ define the same slope. In the case when $\Lambda$ is horizontal, if $w$ and $w'$ define the same slope, we do not have to consider the exponents $c'_i$ because  $x^{\kappa_i}+h_i(x,y)$  does not vanish at the origin.
\medskip

\textbf{Step B.}  Suppose that  $c>\sigma_m$.      
Since $w\left((x^{\alpha}+y^{\beta})^c\right)\leq  w(x^{v_m}y^{v_m})$,  if $\alpha, \beta \geq 2$, 
then we  immediately obtain $c\leq v_m<\sigma_m$. 
Therefore, either $\alpha=1$ or $\beta=1$. By exchanging coordinates if necessary, we may assume that $\alpha=1$.
The exchanging coordinates changes $g$ from $x+y^\nu+\ldots$ into $y+x^\nu+\ldots$. However, it is easy to see  that, with the exchanged $g$, every claim and every step work verbatim in the whole proof.
Note that $\alpha=1$ implies that $w(x)\geq w(y)$ and $\frac{w(x)}{w(y)}$ is the integral number $\beta$.
\medskip

We may therefore write
$$f_w=\epsilon x^ay^b\left(x+A_1y^\beta\right)^{c} \left(x^{\alpha_2}+g_2(x,y)\right)^{c_2} \cdots\left(x^{\alpha_r}+g_r(x,y)\right)^{c_r},$$ 
where $A_1$ is a non-zero constant.
The weighted leading term $f_w$ contains the monomial $x^{(a+c+\sum_{i=2}^r \alpha_ic_i)}y^b$  and  
$a+c+\sum_{i=2}^r \alpha_ic_i\geq c>\sigma_m>v_m$.

\begin{claim}\label{claim:beta}
The exponent $\beta$ is at most $2$.
\end{claim}

This immediately follows from
\[3mn\ell_m\geq\deg(f)\geq \deg (f_w)\geq \beta c>\beta \sigma_m>\beta\left(\frac{4n+5}{4n+4}\right)mn\ell_m.\]

Due to Claims~\ref{claim-horizontal} and~\ref{claim:beta}, we may assume that $\beta$ is either $1$ or $2$.

\medskip

 We apply  a change of coordinates $x+A_1y^\beta\mapsto x$ to the polynomials $f_i$, $f$, $g$ and  $h$.  Set
\[ \begin{split}f_i^{(1)}\left(x,y\right)&:=f_i\left(x-A_1y^\beta,y\right), \\
 f^{(1)}\left(x,y\right)&:=f\left(x-A_1y^\beta,y\right), \\
g^{(1)}\left(x,y\right)&:=g\left(x-A_1y^\beta,y\right), \\
 h^{(1)}\left(x,y\right)& :=h\left(x-A_1y^\beta,y\right).\\ \end{split} \]
Then $f^{(1)}_w\left(x,y\right)=f_w\left(x-A_1y^\beta,y\right)$ and $f^{(1)}=\prod  f^{(1)}_i$.

\begin{claim}
 The Newton polygon of the polynomial $f^{(1)}$  again contains the point $(v_m, v_m)$.
 \end{claim}
 
 This immediately follows from \cite[Lemma~4.3]{ParkWon}.

\begin{claim}
The new polynomial $g^{(1)}$ must contain $y^{\beta}$.
\end{claim}
This is obvious since $\beta\leq 2<\nu$.  
\bigskip

Now we go back to Step~A with $f^{(1)}$, $g^{(1)}$ and $h^{(1)}$ instead of $f$, $g$ and $h$, i.e., let $\Lambda^{(1)}$ be the edge of the Newton polygon of $f^{(1)}$ that intersects the line $s=t$. We also assign  weights $w^{(1)}(x), w^{(1)}(y)$ in such a way that  the monomials of $f^{(1)}_{w^{(1)}}$ lie on the edge $\Lambda^{(1)}$.  
The slope of the  edge $\Lambda^{(1)}$ is~$-\frac{w^{(1)}(x)}{w^{(1)}(y)}$. 

Let $L$ be the line in $\mathbb{R}^2$ obtained by extending the edge $\Lambda$. We observe that 
\[(\dagger)   \left\{\aligned
&\text{there is no monomial of $f^{(1)}$ corresponding to the points under the line $L$;}\\
&\text{there is no monomial of $f^{(1)}$  corresponding to the points on the line $L$  with $s<c$.}\\
\endaligned\right.%
\]
On the other hand,  $f^{(1)}$ contains the monomial $x^{a+c+\sum_{i=2}^r \alpha_ic_i}y^b$ that lies on $L$.
This shows that~$\Lambda^{(1)}$ is  strictly steeper than $\Lambda$,  i.e., $-\frac{w^{(1)}(x)}{w^{(1)}(y)}<-\frac{w(x)}{w(y)}$.

As before, we have an irreducible decomposition  
$$f^{(1)}_{w^{(1)}}=\epsilon^{(1)} x^{a^{(1)}}y^{b^{(1)}}\left(x^{\alpha^{(1)}_1}+g_1^{(1)}(x,y)\right)^{c_1^{(1)}} \cdots\left(x^{\alpha^{(1)}_{r^{(1)}}}+g_{r^{(1)}}^{(1)}(x,y)\right)^{c^{(1)}_{r^{(1)}}},$$ 
where $\epsilon^{(1)}$ is a non-zero constant and $g_i^{(1)}(x,y)$ is a quasi-homogeneous polynomial of degree $w^{(1)}(x^{\alpha_i^{(1)}})$ that does not contain the monomial $x^{\alpha_i^{(1)}}$. 

 Let $c^{(1)}=\max\left\{c_i^{(1)}\right\}.$
 Again we assume that $c^{(1)}=c_1^{(1)}$. The  polynomial $g_1^{(1)}$ must 
  contain  the monomial $y^{\beta^{(1)}}$ for some positive integer $\beta^{(1)}$.

If $c^{(1)}\leq \sigma_m$, then the proof  is done by Claim~\ref{claim:small-c}. If $c^{(1)}> \sigma_m$, then we follow Step~B.  We should here remark that  $\alpha^{(1)}$ must be $1$ 
because 
$\frac{w^{(1)}(x)}{w^{(1)}(y)}>\frac{w(x)}{w(y)}\geq 1$.  Note $\frac{w^{(1)}(x)}{w^{(1)}(y)}$ is the integral number $\beta^{(1)}$ and 
$\frac{w^{(1)}(x)}{w^{(1)}(y)}-\frac{w(x)}{w(y)}=\beta^{(1)}-\beta\geq 1$. Now we go back to Step~A with the newly coordinate-changed polynomials  $f_i^{(2)}$, $f^{(2)}$, $g^{(2)}$ and $h^{(2)}$. 

\begin{claim}
For each $k$, the coordinate-changed polynomial  $f^{(k)}$ satisfies  Claim~\ref{claim:basis-mult}.
\end{claim}
This immediately follows from \cite[Lemma~4.3]{ParkWon}.

\begin{claim}
For each $k$, the coordinate-changed polynomial  $g^{(k)}$ satisfies Claim~\ref{claim:exponent-y}.
\end{claim}
Note that  $\beta\leq 2 n+1$.  For $k\geq 1$, $g^{(k)}$ must keep the monomial~$y^{\beta}$ since $\beta<\beta^{(1)}$ and the sequence $\{\beta^{(i)}\}$ is strictly increasing.

\begin{claim}
Claim~\ref{claim:small-c} is valid for each loop.
\end{claim}

From the second loop, only Case~(a) in Claim~\ref{claim:small-c} may not be valid because it uses Claim~\ref{claim:c-bound}.

Suppose that $\beta=2$. Then, from the second loop, only if $\frac{w'(x)}{w'(y)}=\beta$, Case~(a) can occur. Since $\frac{w^{(k)}(x)}{w^{(k)}(y)}=\beta^{(k)}>\beta$,  the monomials in $f^{(k)}_{w'}$ appear below the line $s=t$. This means that the exponents of $y$ in the monomials of $f^{(k)}_{w'}$  cannot exceed $v_m$. Therefore, $c_i\leq \frac{v_m}{2}$. 

We now suppose that $\beta=1$. Then  $w(x)=w(y)$. Observe that after the first loop, Case~(a) can occur only in the second loop. Moreover, if Case~(a) occurs in the second loop, then~$w'(x)=w'(y)$ and 
\[\begin{split}  f_w &=\epsilon x^ay^b\left(x+A_1y\right)^{c} \left(x+A_2y\right)^{c_2} \cdots\left(x+A_ry\right)^{c_r},\\
 f_{w'}^{(1)} &=f_{w}^{(1)}=\epsilon \left(x-A_1y\right)^ay^bx^{c} \left(x+(A_2-A_1)y\right)^{c_2} \cdots\left(x+(A_r-A_1)y\right)^{c_r},\\
 g_{w'}^{(1)} &=g_{w}^{(1)}=x-A_1y.\\ \end{split}\]

Here $c>\sigma_m$.
Since the Newton polygon of $f$ contains $(v_m, v_m)$, it follows that $a+c$ cannot exceed $2v_m$. Therefore,
\[a+mn\ell_m\leq 2v_m-c+mn\ell_m<2v_m-\sigma_m+mn\ell_m<\frac{2mn\ell_m}{\lambda}.\]
This validates Claim~\ref{claim:small-c}.

\begin{claim}
The whole procedure terminates in a finite number of loops.
\end{claim}
The slope of $\Lambda^{(i)}$ is bounded from  below by $-v_m$ since the Newton polygon of  $f^{(i)}$ must contain the point $(v_m, v_m)$ and it also keeps the property ($\dagger$) above. Termination is therefore guaranteed because the slope of $\Lambda^{(i)}$ drops by at least $1$ for each loop.
\end{proof}

\begin{proof}[Proof of  Theorem~\ref{theorem:main-computation}]
For every $\mathbb{Q}$-divisor $D$ of $2mn$-basis type with respect to $\left(Y_n, \frac{1}{2}C_w\right)$ with $n\geq 4$ and sufficiently large $m$, Theorems~\ref{theorem:smooth-locus} and \ref{theorem:singular-point}  show that
$\left(Y_n, \frac{1}{2}C_w+\frac{8n+8}{8n+7}D\right)$ is log canonical. This implies that for  $n\geq 4$ and sufficiently large $m$
$$\delta_{2mn}\left(Y_n,\frac{1}{2}C_w\right)\geq \frac{8n+8}{8n+7}.$$
This implies that
$$\delta\left(Y_n,\frac{1}{2}C_w\right)\geq \limsup_m \delta_{2mn}\left(Y_n,\frac{1}{2}C_w\right)\geq \frac{8n+8}{8n+7}.$$
\end{proof}

\medskip

\footnotesize

\textbf{Acknowledgements.}

The authors was supported by IBS-R003-D1, Institute for Basic Science in Korea. 
They thank Charles Boyer and Chi Li for their explanations on Sasaki-Einstein metrics and conical K\"ahler-Einstein metrics.


\normalsize

\begin{thebibliography}{99}
\label{section:references}


\bibitem{Barden}
D.~Barden, \emph{Simply connected five-manifolds},  Ann. of Math. (2) \textbf{82} (1965) 365--385.


\bibitem{Berman}
R.~Berman, \emph{K-polystability of $\mathbb{Q}$-Fano varieties admitting K\"ahler-Einstein metrics}, Invent. Math. \textbf{203} (2016), no. 3, 973--1025.

\bibitem{BlumThesis}
H.~Blum, \emph{Singularities and K-stability}, Thesis (Ph.D.) University of Michigan. 2018. 101 pp.

\bibitem{Blum-Jonsson17}
H.~Blum, M.~Jonsson, \emph{Thresholds, valuations, and $K$-stability}, preprint, arXiv:1706.04548.

\bibitem{BG00}
C.~Boyer, K.~Galicki, \emph{On Sasakian-Einstein geometry}, Internat. J. Math. \textbf{11} (2000), no. 7, 873--909. 

\bibitem{BG01}
C.~Boyer, K.~Galicki, \emph{New Einstein metrics in dimension five}, J. Differential Geom. \textbf{57} (2001), no. 3, 443--463.





\bibitem{BG06}
C.~Boyer, K.~Galicki, \emph{Einstein metrics on rational homology spheres}, J. Differential Geom. \textbf{74} (2006), no. 3, 353--362. 

\bibitem{BoyerBook}
C.~Boyer, K.~Galicki, \emph{Sasakian geometry},  Oxford Mathematical Monographs. Oxford University Press, Oxford, 2008. xii+613 pp.


\bibitem{BGN03b}
C.~Boyer, K.~Galicki, M.~Nakamaye, \emph{Sasakian geometry, homotopy spheres and positive Ricci curvature}, Topology \textbf{42} (2003), no. 5, 981--1002. 

\bibitem{BGN03c}
C.~Boyer, K.~Galicki, M.~Nakamaye, \emph{On positive Sasakian geometry}, Geom. Dedicata \textbf{101} (2003), 93--102.
 
\bibitem{BGK05}
C.~Boyer, K.~Galicki, J.~Koll\'ar, \emph{Einstein metrics on spheres},  Ann. of Math.  \textbf{162} (2005), 557--580.


\bibitem{BoNa10}
C.~Boyer,  M.~Nakamaye, \emph{On Sasaki-Einstein manifolds in dimension five}, Geom. Dedicata \textbf{144} (2010), 141--156.



\bibitem{CPS18}
I.~Cheltsov, J.~Park, C.~Shramov, \emph{Delta invariants of singular del Pezzo surfaces}, arXiv:1809.09221.

\bibitem{CRZ18}
I.~Cheltsov, Y.~Rubinstein, K.~Zhang, \emph{Basis log canonical thresholds, local intersection estimates, and asymptotically log del Pezzo surfaces},  Selecta Math. (N.S.)  \textbf{25} (2019), no. 2, 25:34.


\bibitem{CZ}
I.~Cheltsov, K.~Zhang, \emph{Delta invariants of smooth cubic surfaces}, arXiv:1807.08960.

\bibitem{CDS15}
X.-X.~Chen, S.~Donaldson, S.~Sun, \emph{K\"ahler--Einstein metrics on Fano manifolds. I, II, III},
J. Amer. Math. Soc. \textbf{28} (2015), no. 1, 183--197, 199--234, 235--278.





\bibitem{DeKo01}
J.-P.~Demailly, J.~Koll\'ar, \emph{Semi-continuity of complex singularity exponents and  K\"ahler--Einstein metrics on Fano orbifolds}, Ann. Sci. \'Ecole Norm. Sup. \textbf{34} (2001), 525--556.


\bibitem{FK89}
Th.~Friedrich, I.~Kath, \emph{Einstein manifolds of dimension five with small first eigenvalue of the Dirac operator}, J. Differential Geom. \textbf{29} (1989), no. 2, 263--279.



\bibitem{Fujita-Odaka16}
K.~Fujita, Y.~Odaka, \emph{On the $K$-stability of Fano varieties and anticanonical divisors}, Tohoku Math. J. (2) \textbf{70} (2018), no. 4, 511--521.


\bibitem{Kob63}
 S.~Kobayashi, \emph{Topology of positively pinched Kaehler manifolds}, T\^ohoku Math. J. (2) \textbf{15} (1963) 121--139.



\bibitem{Ko97}
J.~Koll\'ar, \emph{Singularities of pairs},  Algebraic geometry--Santa Cruz 1995, 221--287, Proc. Sympos. Pure Math.,  \textbf{62}, Part 1, Amer. Math. Soc., Providence, RI, 1997.


\bibitem{Ko05}
J.~Koll\'ar, \emph{Einstein metrics on five-dimensional Seifert bundles},  J. Geom. Anal. \textbf{15} (2005), no. 3, 445--476.

\bibitem{Ko06}
J.~Koll\'ar, \emph{Positive Sasakian structures on 5-manifolds}, Riemannian topology and geometric structures on manifolds, 93--117, Progr. Math., \textbf{271}, Birkh\"auser Boston, Boston, MA, 2009.


\bibitem{Laz-positivity-in-AG}
R. Lazarsfeld, \emph{Positivity in Algebraic Geometry, II},  Ergebnisse der Mathematik und ihrer Grenzgebiete. 3. Folge. A Series of Modern Surveys in Mathematics [Results in Mathematics and Related Areas. 3rd Series. A Series of Modern Surveys in Mathematics], 49. Springer-Verlag, Berlin, 2004. xviii+385 pp.

\bibitem{LTW17}
C.~Li, G.~Tian, F.~Wang, 
\emph{On Yau--Tian--Donaldson conjecture for singular Fano varieties}, arXiv:1711.09530.

\bibitem{LTW19}
C.~Li, G.~Tian, F.~Wang, 
\emph{The uniform version of Yau-Tian-Donaldson conjecture for singular Fano varieties}, arXiv:1903.01215.

\bibitem{LWX14}
C.~Li, X.~Wang, C.~Xu, 
\emph{On the proper moduli spaces of smoothable K\"ahler-Einstein Fano varieties}, Duke Math. J.
\textbf{168}, no. 8   (2019), 1387--1459.



\bibitem{M68}
J.~Milnor, 
\emph{Singular points of complex hypersurfaces}, Annals of Mathematics Studies, No. \textbf{61} Princeton University Press, Princeton, N.J.; University of Tokyo Press, Tokyo 1968 iii+122 pp.

\bibitem{MO70}
J.~Milnor, P.~Orlik, 
\emph{Isolated singularities defined by weighted homogeneous polynomials}, Topology \textbf{9} (1970), 385--393.


\bibitem{Nadel}
A.~Nadel, 
\emph{Multiplier ideal sheaves and K\"ahler-Einstein metrics of positive scalar curvature}, Ann. of Math. \textbf{132} (1990), 549--596.

\bibitem{ParkWon}
J.~Park, J.~Won, 
\emph{$K$-stability of smooth del Pezzo surfaces},  Math. Ann. \textbf{372} (2018), no. 3--4, 1239--1276.


\bibitem{Reid}
M.~Reid, 
\emph{Canonical 3-folds}, Journ\'ees de G\'eometrie Alg\'ebrique d'Angers, Juillet 1979/Algebraic Geometry, Angers, 1979, pp. 273--310, Sijthoff \& Noordhoff, Alphen aan den Rijn—Germantown, Md., 1980. 


\bibitem{Sav79}
I.V.~ Savel'ev, 
\emph{Structure of the singularities of a class of complex hypersurfaces}, Mat. Zametki \textbf{25} (1979), no. 4, 497--503.

\bibitem{Smale62} 
S.~Smale, 
\emph{On the structure of $5$-manifolds},  Ann. of Math. (2) \textbf{75} (1962) 38--46.

\bibitem{SSY16}
C.~Spotti, S.~Sun, C.~Yao,
\emph{Existence and deformations of K\"ahler–Einstein metrics on smoothable $\mathbb{Q}$-Fano varieties}, Duke Math. J. \textbf{165}, no. 16 (2016), 3043--3083.

\bibitem{Suess} 
H.~S\"u\ss,  
\emph{On irregular Sasaki-Einstein metrics in dimension $5$}, arXiv:1806.00285. 

\bibitem{Takahashi}
 T.~Takahashi, 
 \emph{Deformations of Sasakian structures and its application to the Brieskorn manifolds}, T\^ohoku Math. J. (2) \textbf{30} (1978), no. 1, 37--43.

\bibitem{Tian87}
G.~Tian, 
\emph{On K\"ahler--Einstein metrics on certain K\"ahler manifolds with $c_{1}(M)>0$}, Inv. Math. \textbf{89} (1987), 225--246.


\bibitem{TianYTD}
G.~Tian, 
\emph{K-stability and K\"ahler-Einstein metrics}, Comm. Pure Appl. Math. \textbf{68} (2015), 1085--1156.

\bibitem{TW}
G.~Tian, F.~Wang,
\emph{On the existence of conic K\"ahler-Einstein metrics}, arXiv:1903.12547.



\end{thebibliography}
\end{document}